\documentclass[12pt]{article}
\usepackage{amsmath, amsthm, amssymb}

\newtheorem{theorem}{Theorem}[section]
\newtheorem*{theorem:mcg}{Theorem~\ref{theorem:mcg}}
\newtheorem*{theorem:main}{Theorem~\ref{theorem:main}}
\newtheorem*{theorem:Aut}{Theorem~\ref{theorem:Aut}}
\newtheorem*{theorem:out}{Theorem~\ref{theorem:out}}
\newtheorem*{theorem:10}{Theorem~\ref{theorem:10}}

\newtheorem{lemma}[theorem]{Lemma}
\newtheorem{corollary}[theorem]{Corollary}
\newtheorem{conjecture}[theorem]{Conjecture}
\newtheorem{proposition}[theorem]{Proposition}
\theoremstyle{definition}
\newtheorem{definition}[theorem]{Definition}

\theoremstyle{remark}
\newtheorem{remark}[theorem]{Remark}
\numberwithin{equation}{section}

\DeclareMathOperator{\Mod}{Mod}
\DeclareMathOperator{\Aut}{Aut}
\DeclareMathOperator{\Out}{Out}

\title{$C^1$ actions on the circle of finite index subgroups of $\Mod(\Sigma_g)$, $\Aut(F_n)$, and $\Out(F_n)$}
\author{Kamlesh Parwani}

\begin{document}

\maketitle

\begin{abstract}
Let $\Sigma_{g}$ be a closed, connected, and oriented surface of genus $g \geq 24$ and  let $\Gamma$ be a finite index subgroup of the mapping class group $\Mod(\Sigma_{g})$ that contains the Torelli group $\mathcal{I}(\Sigma_g)$. Then any orientation preserving $C^1$ action of $\Gamma$ on the circle cannot be faithful.

We also show that if $\Gamma$ is a finite index subgroup of $\Aut(F_n)$, when $n \geq 8$, that contains the subgroup of IA-automorphisms, then any orientation preserving $C^1$ action of $\Gamma$ on the circle  cannot be faithful.

Similarly, if $\Gamma$ is a finite index subgroup of $\Out(F_n)$, when $n \geq 8$, that contains the Torelli group $\mathcal{T}_n$, then any orientation preserving $C^1$ action of $\Gamma$ on the circle  cannot be faithful.

In fact, when $n \geq 10$, any orientation preserving $C^1$ action of a finite index subgroup of $\Aut(F_n)$ or  $\Out(F_n)$ on the circle cannot be faithful.
\end{abstract}

\section{Introduction} 
In a previous article \cite{Parwani}, the author had established the following result.
\begin{theorem} [Parwani]
Let $S$ be a connected orientable surface with finitely many punctures, finitely many boundary components, and genus at least $6.$ Then any $C^1$ action of the mapping class group of $S$ on the circle is trivial.
\end{theorem}

In this article, we generalize this result  to  actions of finite index subgroups of mapping class groups that contain the Torelli group.

\begin{theorem} \label{theorem:mcg}
Let $\Sigma_{g}$ be a closed, connected, and oriented surface of genus $g \geq 24$ and  let $\Gamma$ be a finite index subgroup of the mapping class group $\Mod(\Sigma_{g})$ that contains the Torelli group $\mathcal{I}(\Sigma_g)$. Then any orientation preserving $C^1$ action of $\Gamma$ on the circle  cannot be faithful.
\end{theorem}

The results in \cite{KKR} show  that one may drop the assumptions of high genus and the finite index subgroup containing the Torelli group under the assumption of extra regularity. For instance, Corollary 1.3 in \cite{KKR} implies that when a surface has genus at least two (and finitely many punctures and boundary components), then any $C^{1+\tau}$ action, with $\tau > 0$, of a finite index subgroup of the mapping class group of the surface on the circle cannot be faithful. The techniques used do not readily apply to $C^1$ actions.

In \cite{Parwani} it was observed that if Ivanov's conjecture is true, that is, every finite index subgroup of the mapping class group has finite abelianization for $g \geq 3$, then there can be no faithful $C^1$ actions for finite index subgroups of the mapping class group when the genus is at least 6.
Theorem \ref{theorem:mcg} follows directly from recent progress in proving Ivanov's conjecture. More precisely, Theorem \ref{theorem:mcg} uses the established cohomological properties of finite index subgroups of the mapping class groups that contain the Torelli group or its commutator subgroup in \cite{EH} and  the following result, which is proved in Section \ref{section:proofs}.

\begin{theorem} \label{theorem:main}
Let $H$ and $G$ be two finitely generated groups such that $H^1(G, \mathbb{R}) = H^1([G,G], \mathbb{R}) =  H^1(H, \mathbb{R}) = H^1([H,H], \mathbb{R}) = 0$. Then for any $C^1$ orientation preserving action of $H \times G$ on the circle, the induced action of either $H \times 1$  or $1 \times G$ factors through an action of a finite cyclic group.
\end{theorem}

Section \ref{section:preliminaries} contains background material needed for $C^1$ actions on the circle. In Section \ref{section:mcg}, finite index subgroups of mapping class groups that contain the Torelli group are investigated before proving Theorem \ref{theorem:mcg}.
In Section 5 and Section 6 we use the same techniques to prove the following.

\begin{theorem} \label{theorem:Aut}
If $\Gamma$ is a finite index subgroup of $\Aut(F_n)$, when $n \geq 8$, that contains the subgroup of IA-automorphisms, then any orientation preserving $C^1$ action of $\Gamma$ on the circle  cannot be faithful.
\end{theorem}

\begin{theorem} \label{theorem:out}
If $\Gamma$ is a finite index subgroup of $\Out(F_n)$, when $n \geq 8$, that contains the Torelli group $\mathcal{T}_n$, then any orientation preserving $C^1$ action of $\Gamma$ on the circle  cannot be faithful.
\end{theorem}

The assumption of the finite index subgroup containing the subgroup of IA automorphisms or the Torelli group can be dropped at the expense of extra regularity. The authors in \cite{BKK} show that any $C^{1+bv}$ action of a finite index subgroup of $\Aut(F_n)$ or $\Out(F_n)$, with $n \geq 3$, cannot be faithful. The proofs in \cite{BKK} rely heavily on Kopell's Lemma which does not hold in the $C^1$ setting (see \cite{Pixton}). 
To the best of our knowledge, this article contains the only results established (up-to-date) for $C^1$ actions on the circle of finite index subgroups of $\Mod(\Sigma_g)$, $\Aut(F_n)$, and $\Out(F_n)$.

In a recent article, the authors in \cite{MW} were able to show that any $C^1$ action of $\Mod(\Sigma_{g,1})$, where $\Sigma_{g,1}$ is a surface of genus at least 3 and has exactly one marked point (puncture), must be trivial. So, actions of the \textit{full} mapping class group on the circle are fairly well understood, and the results in this paper are the first step towards understanding $C^1$ actions of finite index subgroups of mapping class groups. Progress in this direction appears to be inextricably linked to progress towards showing that finite index subgroups of $\Mod(\Sigma_g)$ have finite abelianiztions.

For arbitrary finite index subgroups in $\Aut(F_n)$ and $\Out(F_n)$, a lot more can be said when $n \geq 10.$ Recent results in \cite{KNO} imply that $\Aut(F_n)$ and $\Out(F_n)$ have Kazhdan's property (T), for all $n \geq 5$. This fact combined with Theorem \ref{theorem:main} immediately leads to the following result.

\begin{theorem} \label{theorem:10}
Any $C^1$ orientation preserving action of a finite index subgroup in $\Aut(F_n)$ or $\Out(F_n)$ cannot be faithful when $n \geq 10.$
\end{theorem}

When $5 \leq n \leq 9$, progress towards showing that finite index subgroup in $\Aut(F_n)$ or $\Out(F_n)$ cannot act faithfully on the circle appears to be linked to generalizing Navas' result in \cite{Navas} to $C^1$ actions. It is not know if $\Aut(F_4)$ or $\Out(F_4)$ have Kazhdan's property (T) or not, and furthermore, the question if a finite index subgroup of $\Aut(F_4)$ or $\Out(F_4)$ has finite ablelianization or not remains open. The best results are contained in \cite{EH}, and these are used in this paper.

Even though we are only able to prove that these actions cannot be faithful, the conclusion of Theorem 1.1 and the results in \cite{MW} suggest that ``the kernel of the action should be large''.

\begin{conjecture}
Let $\Gamma$ be a finite index subgroup of the mapping class group of $S$, where $S$ is a connected, orientable surface with finitely many punctures, finitely many boundary components, and genus at least $3$.  Then any $C^1$ action of $\Gamma$ on the circle factors through an action of a finite cyclic group. Also, the same conclusion holds when $\Gamma$ is a finite index subgroup of $\Aut(F_n)$ or $\Out(F_n)$, with $n \geq 4.$
\end{conjecture}

In other words, finite index subgroups of $\Mod(\Sigma_g)$, $\Aut(F_n)$, and $\Out(F_n)$ should behave like lattices in higher rank Lie groups when considering $C^1$ actions on the circle (see Theorem 1.1 in \cite{Ghys2}, for instance).

\begin{remark}
There exists a finite index subgroup of $\Aut(F_3)$ that has infinite abelianization (see Proposition 9.5 in \cite{EH}). So, finite index subgroups of $\Aut(F_3)$ do not behave like Kazhdan groups  with respect to this property. The results in \cite{BKK} imply that any $C^{1+bv}$ action of this finite index subgroup on the circle cannot be faithful. However, nothing is known about $C^1$ actions of this group on the circle.

Along the same lines, there exists infinitely many level $m$ congruence subgroups $\Mod(\Sigma_2)[m]$ of the mapping class group of a closed and oriented surface of genus two such that  $H^1(\Mod(\Sigma_2)[m], \mathbb{Z})$ is not trivial (see \cite{Mcarthy}). The result in  \cite{KKR} implies that any $C^{1+\tau}$ action, with $\tau >0$, cannot be faithful. Again, nothing is known about $C^1$ actions of these finite index subgroups on the circle.
\end{remark}

\subsection*{Acknowledgements}
Conversations with Andy Putman were invaluable in the preparation of this article. 

\section{Preliminaries} \label{section:preliminaries}

\begin{definition}
By an \textit{action} of a group $G$ we mean a homomorphism $\alpha: G \to \textrm{Diff}_+^1(S^1)$, where $\textrm{Diff}_+^1(S^1)$ is the group of orientation preserving $C^1$-diffeomorphims of the circle. However, we will suppress this notation and simply use $g$ instead of $\alpha(g)$.\end{definition}

We now present several results that will be used in the proof of Theorem~\ref{theorem:mcg} and Theorem~\ref{theorem:main}.

\subsection{Thurston's Stability Lemma}

\begin{theorem}[Thurston \cite{Thurston}] \label{theorem:Thurston}
Let $G$ be a finitely generated group acting on $\mathbb{R}^n$ with a global fixed point $x.$ If the action is $C^1$ and $Dg(x)$ is the identity for all $g \in G$, then either there is a nontrivial homomorphism of $G$ into $\mathbb{R}$ or $G$ acts trivially.
\end{theorem}

We will need the following result which is a direct consequence of Thurston's Stability Lemma (Theorem~\ref{theorem:Thurston}).

\begin{lemma} \label{lemma:trivial}
Let $G$ be a finitely generated group with a $C^1$ orientation preserving action on the circle. If $G$ acts with a global fixed point and $H^1(G, \mathbb{R}) = 0$, then $G$ acts trivially.
\end{lemma}
\begin{proof}
Let $x$ be the global fixed point. It suffices to show that $g'(x) = 1$ for all $g \in G.$ So consider the homomorphism $L : G \to \mathbb{R}$ defined by $L(g) = \log (g'(x))$. Since $H^1(G, \mathbb{R}) = 0$, this must be the trivial homomorphism, which implies that $g'(x) = 1$ for all $g \in G.$
\end{proof}

\subsection{Rotation numbers}

The subject matter of this subsection is well known.  The interested reader may refer to \cite{Ghys}  for more details.

\begin{definition}
Let $G$ be a finitely generated subgroup of the orientation preserving homeomorphisms of the circle and let $\mu$ be a $G$-invariant probability measure.
The \textit{mean rotation number homomorphism} is the map $\rho: G \to \mathbb{R}/\mathbb{Z}$ defined by
\[
g \to \int_{S^1} (\tilde{g} - Id) \, d \tilde{\mu} \, \, (\textrm{Mod} \, 1), 
\]
where $\tilde{g}$ and $\tilde{\mu}$ are lifts of $g$ and $\mu$ to the real line and the integral is over a single fundamental domain. 
\end{definition}

The fact that this map is a homomorphism follows easily from the assumption that $G$ preserves $\mu.$ Note that the mean rotation number of a homeomorphism is the same as the translation number of a circle homeomorphism, which was originally defined by Poincar\'e.

\begin{proposition}
Let $G$ be a finitely generated group with an orientation preserving action on the circle and let $\mu$ be a $G$-invariant probability measure. The element $g \in G$ acts with a fixed point if $\rho(g) = 0.$
\end{proposition}

\section{$C^1$ actions on the circle} \label{section:proofs}

\subsection{Action of the commutator subgroup}

An immediate consequence of the previous proposition is the following observation.
\begin{corollary}
Let $G$ be a finitely generated group with an orientation preserving action on the circle and let $\mu$ be a $G$-invariant probability measure. If $H_1(G,\mathbb{Z}) = 0$, then $G$ acts with a global fixed point.
\end{corollary}
\begin{proof}
Since $G$ is perfect and the target of  $\rho$, the rotation number homomorphism, is an abelian group, $\rho$ must be trivial. The  proposition above implies that every element acts with a fixed point. The support of $\mu$ is contained in the fixed point set of each element in $G$, and therefore, $G$ must have a global fixed point.
\end{proof}

\begin{corollary} \label{corollary:finite}
Let $G$ be a finitely generated group with an orientation preserving action on the circle and let $\mu$ be a $G$-invariant probability measure. If $H^1(G,\mathbb{R}) = 0$, then the commutator subgroup $[G,G]$ has a global fixed point. 
\end{corollary}
\begin{proof}
The commutator subgroup plays the role of the perfect group in the proof of the corollary above. For every $g \in [G, G]$, $\rho(g) = 0.$  The proposition above implies that every element in $[G,G]$ acts with a fixed point. Recall that the support of the invariant measure $\mu$ is contained in the fixed point set of every element in $G.$ So, the commutator subgroup acts with a global fixed point.
\end{proof}

\begin{lemma} \label{lemma:apply}
Let $G$ be a finitely generated group with an orientation preserving action on the circle and let $\mu$ be a $G$-invariant probability measure. If $H^1(G, \mathbb{R}) = H^1([G,G],\mathbb{R}) = 0$, the action of $G$ factors through an action of a finite cyclic group.
\end{lemma}
\begin{proof}
The arguments above imply that the set of global fixed points for the commutator $[G, G]$ is not empty. Also note that $[G, G]$ is a finite index subgroup of a finitely generated group, and so, it too is finitely generated. Now, Lemma \ref{lemma:trivial} implies that the action of the commutator subgroup is trivial. So, the action of $G$ factors through an action of $G/[G, G]$, a finite abelian group. Since every finite group acting on the circle that preserves orientation is cyclic (see \cite{Ghys} and other reference therein), the action of $G$ factors through an action of a finite cyclic group.
\end{proof}

\subsection{Hyperbolic fixed points}

The proof of Theorem~\ref{theorem:main} relies heavily on the following result, which guarantees the existence of an element with a finite fixed point set in the absence of an invariant probability measure.

\begin{theorem}[Deroin, Kleptsyn, and Navas \cite{DNK}]
Let $G$ be a countable group with an orientation preserving $C^1$ action on the circle.  If there is no $G$-invariant probability measure for the action, then there exists an element $g \in G$ that only has hyperbolic fixed points. In particular, $g$ has a nonempty finite set of fixed points. 
\end{theorem}

\subsection{Action of products}

The arguments in this section are closely related to the ones in \cite{Parwani} where only perfect groups are considered. 
\begin{theorem:main}
Let $H$ and $G$ be two finitely generated groups such that $H^1(G, \mathbb{R}) = H^1([G,G], \mathbb{R}) =  H^1(H, \mathbb{R}) = H^1([H,H], \mathbb{R}) = 0$. Then for any $C^1$ orientation preserving action of $H \times G$ on the circle, the induced action of either $H \times 1$  or $1 \times G$ factors through an action of a finite cyclic group.
\end{theorem:main}
\begin{proof}
Suppose that there exists a $C^1$ orientation preserving action of $H \times G$ on the circle.  We now consider the induced action of $G$ on the circle. By the result of Deroin, Kleptsyn, and Navas, either there exists a probability measure $\mu$ that is $G$-invariant or there is an element  $g \in G$ that has a finite number of fixed points. We treat these two cases separately.

\smallskip

\noindent
CASE 1: There exists a probability measure $\mu$ that is $G$-invariant.

Lemma \ref{lemma:apply} implies that the action of $G$ factors through an action of a finite cyclic group.

\smallskip

\noindent
CASE 2:  There is an element $g$ in $G$ that has a finite number of fixed points.

In this case, there is a finite set---the set of hyperbolic fixed points---left invariant by the induced action of the group $H.$ This implies that the action of the group $H$ has an invariant measure, and now we may argue as above, with $G$ replaced by $H$, to conclude that the action of  $H$ factors through an action of a finite cyclic group.
\end{proof}

\section{Finite index subgroups of mapping class groups} \label{section:mcg}

We refer the reader to the excelent survey article \cite{Ivanov} or to the text \cite{FM} for an introduction to mapping class groups and to the article \cite{Johnson2}  by D. Johnson for the technical aspects of the Torelli group.

Let $\Sigma_{g}$ be a closed, connected,  and oriented surface of genus $g$.  The Torelli group $\mathcal{I}(\Sigma_g)$ is the kernel of the symplectic action of $\Mod(\Sigma_{g})$ on $H_1(\Sigma_{g}, \mathbb{Z}).$ Thus, $\mathcal{I}(\Sigma_g)$ is defined by
\[
1 \to \mathcal{I}(\Sigma_g) \to \Mod(\Sigma_g) \to Sp(2g, \mathbb{Z}) \to 1
\]
Johnson in \cite{Johnson2} showed that when $g \geq 3$, $\mathcal{I}(\Sigma_g)$ is  generated by bounding pair maps. In the same article, he also proved that $\mathcal{I}(\Sigma_g^1)$ is generated by bounding pair maps, where $\Sigma_g^1$ is the compact oriented surface of genus $g \geq 3$ with exactly one boundary component.

The following results will be crucial to revealing the cohomological properties of finite index subgroups of the mapping class group that contain the Torelli group.

\begin{theorem}[Hain \cite{Hain}]
Let $\Sigma_{g}^n$ be a compact and oriented surface of genus $g \geq 3$ and $n \geq 0$ boundary components. Let $\Gamma$ be a finite index subgroup of the mapping class group $\Mod(\Sigma_{g}^n)$ that contains the Torelli subgroup $\mathcal{I}(\Sigma_g^n)$. Then $H^1(\Gamma, \mathbb{Z}) = 0.$
\end{theorem}

The next result is contained in the statement of Theorem 1.9 in \cite{EH}. The authors in \cite{EH} prove a more general theorem establishing finite abelianizations of subgroups of finite index containing groups in the Johnson filtration. We only require a special case of their result.

\begin{theorem}[Ershov and He \cite{EH}]
Let $\Sigma_{g}^1$ be a compact and oriented surface of genus $g \geq 12$ and exactly one boundary component. Let $\Gamma$ be a finite index subgroup of the mapping class group $\Mod(\Sigma_{g}^1)$  that contains $[\mathcal{I}(\Sigma_g^1), \mathcal{I}(\Sigma_g^1)]$. Then $\Gamma$ has finite abelianization. In particular, $H^1(\Gamma, \mathbb{R}) = 0.$
\end{theorem}

Let $\Sigma_{g}$ be a closed, connected,  and oriented surface of genus $g \geq 24$ and let $\Gamma$ be a finite index subgroup of the mapping class group $\Mod(\Sigma_{g})$ that contains the Torelli group $\mathcal{I}(\Sigma_g)$.  Now, since the genus of the surface is at least 24, there exists a simple closed separating curve that splits $\Sigma_{g}$ into two subsurfaces $\Sigma_{g_1}^{1}$ and $\Sigma_{g_2}^{1}$, where $g_1$ and $g_2$ are at least 12 and the (sub)surfaces have exactly one boundary component. Let $G = \Mod(\Sigma_{g_1}^{1}) \cap \Gamma$ and let $H = \Mod(\Sigma_{g_2}^{1}) \cap \Gamma.$

\begin{lemma}
The groups $G$ and $H$ are finitely generated and $H^1(G, \mathbb{R}) = H^1([G,G], \mathbb{R}) \\ = H^1(H, \mathbb{R}) = H^1([H,H], \mathbb{R}) = 0.$
\end{lemma}
\begin{proof}
These results will only be established for $G$, and then the same arguments will apply in the case of $H$. First note that since $\Gamma$ has finite index in $\Mod(\Sigma_{g})$, $G = \Mod(\Sigma_{g_1}^{1}) \cap \Gamma$ has finite index in  $\Mod(\Sigma_{g_1}^{1}).$ This also implies that $G$ is finitely generated (since mapping class groups are finitely generated). Also, because $\Gamma$ contains $\mathcal{I}(\Sigma_g)$, it contains all the generating bounding pair maps. This means that $\Mod(\Sigma_{g_1}^{1}) \cap \Gamma$ contains all the bounding pair maps in $\Sigma_{g_1}^{1}.$ Since the bounding pair maps generate the Torelli group $\mathcal{I}(\Sigma_{g_1}^{1})$ when the genus is at least 3 (see \cite{Johnson2}), $G$ contains the Torrelli group $\mathcal{I}(\Sigma_{g_1}^{1})$. So, $G$ is a finite index subgroup of $\Mod(\Sigma_{g_1}^{1})$ that contains its Torelli group. By the result of Hain, $H^1(G, \mathbb{Z}) = 0.$ It now follows that  $H^1(G, \mathbb{R}) = 0.$

Since $H^1(G, \mathbb{R}) = 0,$ the commutator subgroup $[G, G]$ is a finite index subgroup of $G.$ Furthermore, $[G, G]$ contains $[I(\Sigma_{g_1}^{1}),  I(\Sigma_{g_1}^{1})]$ because $G$ contains $ I(\Sigma_{g_1}^{1}).$ Hence, $[G, G]$ is a finite index subgroup of $\Mod(\Sigma_{g_1}^{1})$ that contains $[I(\Sigma_{g_1}^{1}),  I(\Sigma_{g_1}^{1})]$, and now the result of Ershov and He implies that $H^1([G,G], \mathbb{R}) = 0.$
\end{proof}

\begin{theorem:mcg}
Let $\Sigma_{g}$ be a closed, connected, and oriented surface of genus $g \geq 24$ and  let $\Gamma$ be a finite index subgroup of the mapping class group $\Mod(\Sigma_{g})$ that contains the Torelli group $\mathcal{I}(\Sigma_g)$. Then any orientation preserving $C^1$ action of $\Gamma$ on the circle cannot be faithful.
\end{theorem:mcg}
\begin{proof}
We may split $\Sigma_{g}$ into two subsurfaces $\Sigma_{g_1}^{1}$ and $\Sigma_{g_2}^{1}$, where $g_1$ and $g_2$ are at least 12. Let $G = \Mod(\Sigma_{g_1}^{1}) \cap \Gamma$ and let $H = \Mod(\Sigma_{g_2}^{1}) \cap \Gamma.$ The lemma above implies that $H^1(G, \mathbb{R}) = H^1([G,G], \mathbb{R}) = H^1(H, \mathbb{R}) =  H^1([H,H], \mathbb{R}) = 0.$ Now, the arguments in Theorem \ref{theorem:main} show that either $[G, G]$ or $[H, H]$ acts trivially. In both cases, the action on the circle cannot be faithful.
\end{proof}

\begin{remark}
The kernel of the action is ``large'' in the sense that both $[G, G]$ and $[H, H]$ are finite index subgroups of the mapping class group of a surface with genus at least 12 and one boundary component.
\end{remark}

The short exact sequence 
\[
1 \to \mathcal{I}(\Sigma_g) \to \Mod(\Sigma_g) \xrightarrow{\Psi} Sp(2g, \mathbb{Z}) \to 1
\]
suggests a simple way to manufacture finite index subgroups in $\Mod(\Sigma_g)$---simply ``pullback'' a finite index subgroup in $Sp(2g, \mathbb{Z}).$ Any finite index subgroup constructed in this manner will contain $\mathcal{I}(\Sigma_g).$

The level $m$ congruence subgroup $Sp(2g, \mathbb{Z})[m]$ of $Sp(2g, \mathbb{Z})$ is defined to be the kernel of the reduction homomorphism Mod $m$ (when $m \geq 2$):
\[
Sp(2g, \mathbb{Z})[m] = ker(Sp(2g, \mathbb{Z}) \to Sp(2g, \mathbb{Z}/m\mathbb{Z})).
\]
The level $m$ congruence subgroup $\Mod(\Sigma_g)[m]$ of $\Mod(\Sigma_g)$, with $m \geq 2$ and $g \geq 1$, is defined to be the preimage $\Psi^{-1} (Sp(2g, \mathbb{Z})[m]).$  

This gives us an immediate corollary of Theorem \ref{theorem:mcg}.

\begin{corollary}
Let $\Sigma_{g}$ be a closed, connected, and oriented surface of genus $g \geq 24.$ Any orientation preserving $C^1$ action of the level $m$ congruence subgroup $\Mod(\Sigma_g)[m]$ on the circle cannot be faithful.
\end{corollary}

\begin{remark}
Another class of finite index subgroups of $\Mod(\Sigma_g)$ that contain $\mathcal{I}(\Sigma_g)$ are the \textit{spin mapping class groups} (see \cite{Harer}).  Theorem \ref{theorem:mcg} applies to these finite index subgroups as well.
\end{remark}

\section{Finite index subgroups of $\Aut(F_n)$}

We refer the reader to \cite{V}, and references therein, for an introduction to $\Aut(F_n)$ and $\Out(F_n).$

Consider the group of automorphisms $\Aut(F_n)$ of the free group $F_n$ of rank $n \geq 8.$ Let $\Gamma$ be a finite index subgroup of $\Aut(F_n)$ that contains the group of IA-automorphisms $IA_n$, the subgroup of automorphisms of $F_n$ which act trivially on its 
abelianization (``Identity on Abelianization" automorphisms). Thus, $IA_n$ is defined by
\[
1 \to IA_n \to \Aut(F_n) \to SL(n, \mathbb{Z}) \to 1
\]

Now, $\Aut(F_n)$ contains $\Aut(F_4) \times \Aut(F_{n-4})$, where $n- 4 \geq 4.$ Let $G = \Gamma \cap \Aut(F_4)$ and $H = \Gamma \cap \Aut(F_{n-4})$.

The following result is a special case of  Theorem 9.4 in \cite{EH}.

\begin{theorem}[Ershov and He \cite{EH}]
Let $n \geq 3$ and let $\Gamma$ be a finite index subgroup of $\Aut(F_n)$ such that $[\Gamma, \Gamma]$ contains $[IA_n, IA_n ]$. Then $\Gamma$ has finite abelianization. In particular, $H^1(\Gamma, \mathbb{R}) = 0.$
\end{theorem}

The authors in \cite{EH} also prove the following (see Theorem 1.9 in \cite{EH}).

\begin{theorem}[Ershov and He \cite{EH}]
Let $n \geq 4$ and let $\Gamma$ be a finite index subgroup of $\Aut(F_n)$ such that $\Gamma$ contains $[IA_n, IA_n ]$. Then $\Gamma$ has finite abelianization. In particular, $H^1(\Gamma, \mathbb{R}) = 0.$
\end{theorem}

As an immediate corollary of these two results is the following observation.

\begin{corollary}
Let $n \geq 4$ and let $\Gamma$ be a finite index subgroup of $\Aut(F_n)$ such that $\Gamma$ contains $IA_n$. Then $\Gamma$ and $[\Gamma, \Gamma]$ both have finite abelianizations. In particular, $H^1(\Gamma, \mathbb{R}) = H^1([\Gamma, \Gamma], \mathbb{R}) =0.$
\end{corollary}

\begin{lemma} \label{lemma:trivial2}
The groups $G$ and $H$ are finitely generated and $H^1(G, \mathbb{R}) = H^1([G,G], \mathbb{R}) \\ =   H^1(H, \mathbb{R}) = H^1([H,H], \mathbb{R}) = 0.$
\end{lemma}
\begin{proof}
These results will only be established for $G$, and then the same arguments will apply in the case of $H$. First note that since $\Gamma$ has finite index in $\Aut(F_n)$, $G = \Aut(F_4) \cap \Gamma$ has finite index in  $\Aut(F_4).$ This also implies that $G$ is finitely generated (since $\Aut(F_n)$ is finitely generated \cite{Nielsen}). Since any automorphism in $\Aut(F_4)$ that acts trivially on its abelianization can be be extended to an automorphism of $\Aut(F_n)$ that acts trivially on the abelianization of $\Aut(F_n)$, $G = \Aut(F_4) \cap \Gamma$ contains $IA_4.$ Now, the corollary above establishes the desired result.
\end{proof}

\begin{theorem:Aut}
If $\Gamma$ is a finite index subgroup of $\Aut(F_n)$, when $n \geq 8$, that contains the subgroup of IA-automorphisms, then any orientation preserving $C^1$ action of $\Gamma$ on the circle  cannot be faithful.
\end{theorem:Aut}
\begin{proof}
The group $\Aut(F_n)$ contains $\Aut(F_4) \times \Aut(F_{n-4})$, where $n- 4 \geq 4.$ Let $G = \Gamma \cap \Aut(F_4)$ and $H = \Gamma \cap \Aut(F_{n-4})$.
 The lemma above implies that $H^1(G, \mathbb{R}) = H^1([G,G], \mathbb{R}) = H^1(H, \mathbb{R}) = 0 = H^1([H,H], \mathbb{R}).$ Now, the arguments in Theorem \ref{theorem:main} show that either $[G, G]$ or $[H, H]$ acts trivially. In both cases, the action on the circle cannot be faithful.
\end{proof}

\begin{remark} \label{remark}
The kernel of the action is ``large'' in the sense that both $[G, G]$ and $[H, H]$ are finite index subgroups of $\Aut(F_k)$,  with $k \geq 4$. \end{remark}

The level $m$ congruence subgroup of $\Aut(F_n)$, for $m \geq 2$, is defined as
\[
\Aut(F_n, m) = ker(\Aut(F_n) \xrightarrow{\Psi} SL(n, \mathbb{Z}) \xrightarrow{\Mod \, m} SL(n, \mathbb{Z}/m\mathbb{Z})).
\]

This gives us an immediate corollary of Theorem \ref{theorem:Aut}.

\begin{corollary}
Any orientation preserving $C^1$ action of the level $m$ congruence subgroup $\Aut(F_n, m)$  on the circle cannot be faithful when $n \geq 8$.
\end{corollary}

\section{Finite index subgroups of $\Out(F_n)$}

Consider the group of outer automorphisms $\Out(F_n)$ of the free group $F_n$ of rank $n \geq 8.$ Let $\Gamma$ be a finite index subgroup of $\Out(F_n)$ that contains the Torelli  group $\mathcal{T}_n$, the subgroup of outer automorphisms of $F_n$ which act trivially on its 
abelianization. Thus, $\mathcal{T}_n$ is defined by
\[
1 \to \mathcal{T}_n \to \Out(F_n) \to SL(n, \mathbb{Z}) \to 1
\]

\begin{theorem:out} 
If $\Gamma$ is a finite index subgroup of $\Out(F_n)$, when $n \geq 8$, that contains the Torelli group $\mathcal{T}_n$, then any orientation preserving $C^1$ action of $\Gamma$ on the circle  cannot be faithful.
\end{theorem:out}
\begin{proof}
Since $\Out(F_n)$ is defined to be the quotient of $\Aut(F_n)$ by the subgroup of inner automorphisms of $F_n$, there is a natural projection $p$ from $\Aut(F_n)$ onto $\Out(F_n).$ If $\Gamma$ is a finite index subgroup of $\Out(F_n)$ that contains the Torelli group $\mathcal{T}_n$, then the ``pullback'' of $\Gamma$ is a finite index subgroup $\Gamma_A$ of $\Aut(F_n)$ than contains $IA_n.$ So, an action of $\Gamma$ on the circle induces an action of $\Gamma_A$ on the circle. 

The group $\Aut(F_n)$ contains $\Aut(F_4) \times \Aut(F_{n-4})$, where $n- 4 \geq 4.$ Let $G_A = \Gamma_A \cap \Aut(F_4)$ and $H_A = \Gamma_A \cap \Aut(F_{n-4})$. Now define $G = p(G_A)$ and $H = p(H_A).$ First observe that since $H^1(G_A, \mathbb{R}) = 0$ (by Lemma \ref{lemma:trivial2} above), $H^1(G, \mathbb{R}) = 0.$ Similarly, $H^1([G_A,G_A], \mathbb{R}) = 0$ implies that $H^1([G,G], \mathbb{R}) = 0$ as well. Also, these same arguments establish that  $H^1(H, \mathbb{R}) = H^1([H, H], \mathbb{R}) = 0.$ Furthermore, the groups $G$ and $H$ are finitely generated because they are homomorphic images of finitely generated groups.

Now, the arguments in Theorem \ref{theorem:main} show that either $[G, G]$ or $[H, H]$ acts trivially. In both cases, the action on the circle cannot be faithful.\end{proof}

\section{Finite index subgroups of $\Aut(F_n)$ and $\Out(F_n)$, with $n \geq 10$}

Recent results  in \cite{KNO} imply that $\Aut(F_n)$ and $\Out(F_n)$ have Kazhdan's property (T), for all $n\geq5.$ Now, a theorem of Navas establishes that when $\alpha > 1/2$, any $C^{1+\alpha}$ orientation preserving action of a finite index subgroup of $\Aut(F_n)$ or $\Out(F_n)$ must factor through an action of a finite cyclic group when $n \geq 5$ (see \cite{Navas}). Again, nothing is known about $C^1$ actions of \textit{arbitrary} finite index subgroups of $\Aut(F_n)$ and $\Out(F_n)$ on the circle, for $5 \leq n \leq 9$. However, Theorem \ref{theorem:main} implies that any $C^1$ orientation preserving action of a finite index subgroup in $\Aut(F_n)$ or $\Out(F_n)$, when $n \geq 10$, cannot be faithful.

\begin{theorem:10}
Any $C^1$ orientation preserving action of a finite index subgroup in $\Aut(F_n)$ or $\Out(F_n)$ cannot be faithful when $n \geq 10.$
\end{theorem:10}
\begin{proof}
Let $\Gamma$ be a finite index subgroup of $\Aut(F_n)$, with $n \geq 10.$ Now, $\Aut(F_5) \times \Aut(F_{n-5})$ is a subgroup of $\Aut(F_n)$, and so, let $G = \Gamma \cap \Aut(F_5)$ and similarly define $H$ as the intersection of $\Gamma$ with $\Aut(F_{n-5}).$ Results in \cite{KNO} imply that $G$ and $H$ both have Kazhdan's property (T). Since $G$ and $H$ are Kazhdan, it immediately follows that the groups $G$ and $H$ are finitely generated and $H^1(G, \mathbb{R}) = H^1([G,G], \mathbb{R}) =   H^1(H, \mathbb{R}) = H^1([H,H], \mathbb{R}) = 0.$ We may now apply the Theorem \ref{theorem:main} to conclude that the action of $\Gamma$ cannot be faithful. An identical argument establishes the same conclusion when $\Gamma$ is a finite index subgroup of $\Out(F_n)$, for $n \geq 10.$
\end{proof}

\end{document}